\documentclass[a4paper,12pt,reqno]{amsart}

\usepackage{amsmath,amssymb,latexsym,amstext,enumerate}
\usepackage{dsfont}
\usepackage{longtable}

\theoremstyle{definition}
\newtheorem{defi}{Definition}[section]

\theoremstyle{plain}

\newtheorem{thm}[defi]{Theorem}

\newtheorem{lem}[defi]{Lemma}
\newtheorem{prop}[defi]{Proposition}
\newtheorem*{prop*} {Proposition}

\theoremstyle{remark}
\newtheorem{rmk}[defi]{Remark}
\newtheorem{nota}[defi]{Notation}

\usepackage{setspace}

\usepackage{bm}

\numberwithin{equation}{section}

\newcommand{\1}{ \mathds{1}}
\newcommand{\comm}[1]{ }

\DeclareMathOperator{\Aut}{Aut}

\newcommand{\Z}{\mathbb{Z}}

\newcommand{\R}{\mathbb{R}}
\newcommand{\la}{\langle}
\newcommand{\ra}{\rangle}
\newcommand{\C}{\mathbb{C}}

\newcommand{\nop}{\mbox{$\circ\atop\circ$}}
\newcommand{\h}{\mathfrak{h}}
\newcommand{\g}{\mathfrak{g}}

\newcommand{\Inn}{\mathrm{Inn}\,}

\newcommand{\rank}{\mathrm{rank}}
\newcommand{\Stab}{\mathrm{Stab}}

\DeclareMathOperator{\Ker}{Ker}
\DeclareMathOperator{\Hom}{Hom}

\DeclareMathOperator{\Span}{Span}
\DeclareMathOperator{\irr}{Irr}
\DeclareMathOperator{\Imm}{Im}

\begin{document}

\title[Unitary forms and holomorphic VOA]{Unitary forms for holomorphic vertex operator algebras of central charge $24$}

 \author{Ching Hung Lam}
 \address[C.H. Lam]{ Institute of Mathematics, Academia Sinica, Taipei  10617, Taiwan}
\email{chlam@math.sinica.edu.tw}

\thanks{C.H. Lam was partially supported by a research grant AS-IA-107-M02 of Academia Sinica  and MoST grants  110-2115-M-001-011-MY3 of Taiwan}

\subjclass[2010]{Primary 17B69; Secondary 20B25 }


\begin{abstract}
We prove that all holomorphic vertex operator algebras of central charge $24$ with non-trivial weight one subspaces are unitary.   The main method is to use the orbifold construction of a holomorphic VOA $V$ of central charge $24$ directly from a Niemeier lattice VOA $V_N$.  We show that it is possible to extend the unitary form for the lattice VOA $V_N$ to the holomorphic VOA $V$ by using the orbifold construction and some information of the automorphism group $\Aut(V)$. 
\end{abstract}
\maketitle


\section{Introduction}
The classification 
of strongly regular holomorphic vertex operator algebras (abbreviated as VOA) of central 
charge $24$  with non-trivial weight one space has recently been completed (see \cite{ELMS,LS20,MS} and the references given there). 
Except for the uniqueness of holomorphic VOAs of moonshine type (i.e, with $V_1=0$), it was proved that there are exactly $70$ strongly regular holomorphic VOAs  
with central charge $24$ and non-zero weight one space; moreover, their VOA structures are uniquely determined by the Lie algebra structures of their weight one spaces. The possible Lie algebra structures for their weight one subspaces are exactly those given 
in Schellekens' list \cite{Sch}.  It is commonly believed that all holomorphic VOAs of central charge $24$ are unitary (i.e., they have some positive definite invariant Hermitian forms). In this article, we show that  all holomorphic vertex operator algebras of central charge $24$ with non-trivial weight one subspaces are unitary. 

It is well known \cite{Bo,Dlin,FLM} that lattice VOAs are unitary. It turns out that many automorphisms of finite order will preserve the unitary form. Our main method is to use the orbifold construction of holomorphic VOAs of central charge $24$ directly from Niemeier lattice VOAs (cf. \cite{HM}).   In \cite{ELMS}, it is proved that any holomorphic VOA of central charge $24$ with a semisimple weight one Lie algebra can be constructed by a single orbifold construction from the Leech lattice VOA $V_{\Lambda}$.  It is well known \cite{Bo,DN} that any automorphism $\tilde{g}\in \Aut(V_{\Lambda})$ can be written as $\tilde{g}=\widehat{\tau}\exp(2\pi i\beta(0))$  
where $\tau\in Co.0=O(\Lambda)$, $\beta\in \R\Lambda^{\tau}$ and $\widehat{\tau}$ denotes a standard lift of $\tau$ in $O(\widehat{\Lambda})$. It was first observed by G. H\"ohn \cite{Ho2}  that the isometry $\tau$ belongs to only $11$ special conjugacy classes in $Co.0=O(\Lambda)$. All these isometries have positive frame shape and 
their fixed point sublattices satisfy some duality properties \cite{LM,MS}.  

In \cite{HM}, other orbifold constructions for holomorphic VOAs of central charge $24$ are discussed. In particular, it was proved that for any  holomorphic VOA $V$ of central charge $24$ with $V_1\neq 0$,  there exist a Niemeier lattice $N$ and an automorphism $g\in \Aut(V_N)$ of finite order such that the VOA   $ \widetilde{V_N}(g)$ obtained by an orbifold construction  from $V_N$ and $g$ is isomorphic to $V$.  Therefore, $V$ contains a subVOA $V_N^g$, which is also unitary. 
In addition, we verify that the irreducible $g^i$-twisted modules of $V_N$ are unitary twisted $V_N$-modules for all $i\in \Z$. By using some information about the automorphism group of $V$, we will show that the unitary form on $V_N^g$ can be extended to $V$ and $V$ itself is also unitary (cf. Theorem \ref{Thm:uni}).  An advantage of using Niemeier lattice VOA is that the order of $g$ can be chosen to be relatively small and $(V_N^g)_1$ is a relatively large Lie subalgebra of $V_1$.  Up to conjugation by an inner automorphism, any automorphism $g\in \Aut(V_N)$ can be written as $g=\hat{\sigma}\exp(2\pi i \gamma(0))$, where $\sigma\in O(N)$, $\gamma\in \mathbb{Q}\otimes_\Z N^\sigma$ and $N^\sigma$ denotes the sublattice of $N$ fixed by $\sigma$. It turns out that $\sigma$ can be chosen such that it has the same frame shape as one of the $11$ conjugacy classes of $Co_0$  discussed above. Moreover, the order of  $g$ is the same as the order of $\sigma$.  
  
The organization of this article is as follows. In Section 2, we  review some basic notions about unitary VOAs and their unitary modules from \cite{Dlin} and \cite{CKLW}. In Section 3, we recall some facts about lattice VOAs and their unitary structures. We also  show that any irreducible $g$-twisted module $V_L^\chi(g)$ for a lattice VOA $V_L$ for a finite order automorphism $g$ is a unitary  $g$-twisted module  for $V_L$.  In Section 4, we first review  several facts about the automorphism groups of holomorphic VOAs of central charge $24$ with non-trivial weight one spaces. We then discuss the orbifold constructions of holomorphic VOAs of central charge $24$ directly from  Niemeier lattice VOAs. Some explicit choices for the Niemeier lattice $N$ and the automorphism $g$ are also discussed.  In Section 5, we study the unitary form for holomorphic VOAs of central charge $24$ with non-trivial weight one spaces.  The main theorem is  Theorem \ref{Thm:uni}. We show that a VOA $V$ is unitary if it contains a pair of commuting automorphisms $(f,h)$ satisfying some conditions.  Finally we discuss a method to define the pair $(f,g)$ for each holomorphic VOA of  central charge $24$ with  non-trivial weight one space.   
\medskip

\paragraph{\bf Acknowledgment.} After this work has been completed, we noticed the preprint by Carpi et. al. \cite{CGGH}, in which the unitary of strongly rational holomorphic vertex operator algebras with central charge $24$ and non-zero weight one subspace is proved; nevertheless, their method uses the theory of tensor category and is quite different from our approach. 

\section{Unitary VOA and unitary modules}

We first recall the notion of unitary VOAs and unitary modules from \cite{Dlin} (see also \cite{CKLW}).


\begin{defi}[\cite{Dlin}] Let $(V, Y, \1,\omega)$ be a vertex operator algebra and let $\phi:V\to V$ be an anti-linear involution of $V$ (i.e, $\phi(\lambda u)= \bar{\lambda} \phi(u)$, $\phi(\1) = \1, \phi(\omega) =\omega$, $\phi(u_n v) = \phi(u)_n \phi(v)$ for any $u, v \in  V$, $n\in \Z$, and $\phi$ has order $2$). Then $(V, \phi)$ is said to be {\it unitary} if there exists a positive-definite Hermitian form $(\ ,\ )_V : V \times V \to \C$, which is $\C$-linear on the first vector and anti-$\C$-linear on the second vector, such that the following invariant property holds for any $a, u, v\in V$:
\[
(Y (e^{zL(1)} (-z^{-2})^{L(0)} a, z^{-1})u, v)_V = (u, Y (\phi(a), z)v)_V,
\]
where $L(n)$ is defined by $Y (\omega, z) = \sum_{n\in \Z} L(n)z^{-n-2}$.
\end{defi}

\begin{rmk}\label{Rem:inv}
By \cite[Proposition 5.3]{CKLW}, $V$ is self-dual and of CFT-type if $(V,\phi)$ is a simple unitary VOA with an invariant Hermitian form $(\cdot,\cdot)_V$. In this case, $V$ has a unique invariant symmetric bilinear form $\langle\cdot,\cdot\rangle$, up to scalar (\cite{L94}).
Normalizing $(\1,\1)_V=\langle\1,\1\rangle=1$, we obtain $(u,v)_V=\langle u,\phi(v)\rangle$ for all $u,v\in V$.
Note that $(\phi(u),\phi(v))_V=\overline{(u,v)}_V=(v,u)_V$ for $u,v\in V$.
\end{rmk}
\begin{defi}[\cite{Dlin}]
Let $(V,\phi)$ be a unitary VOA and $g$ a finite order automorphism of $V$. 
An (ordinary) $g$-twisted $V$-module $(M, Y_M)$ is called a {\it unitary $g$-twisted $V$-module} if there exists
a positive-definite Hermitian form $(\ ,\  )_M : M \times M \to  \C$ such that the following
invariant property holds for $a\in V$ and $w_1,w_2 \in M$:
\begin{equation}
(Y_M(e^{zL(1)} (-z^{-2})^{L(0)} a, z^{-1})w_1, w_2)_M = (w_1, Y_M(\phi(a), z)w_2)_M.\label{Eq:Inv}
\end{equation}
We call such a form {\it a positive-definite invariant Hermitian form}.
\end{defi}

The following lemma follows from the similar argument as in \cite[Remark 5.3.3]{FHL}.

\begin{lem}[{cf.\ \cite[Remark 5.3.3]{FHL}}]\label{Rem:specify}
Let $(V,\phi)$ be a unitary VOA.
Let $M$ be a $V$-module and let $M'$ be the contragredient module of $M$ with a natural pairing $\langle\cdot,\cdot\rangle$ between $M$ and $M'$. 
\begin{enumerate}
\item If $M$ has a non-degenerate invariant sesquilinear form $(\cdot,\cdot)$, which is linear on the first vector and anti-$\C$-linear on the second vector and satisfies the invariant property \eqref{Eq:Inv},  then the map $\Phi:M\to M'$ defined by $(u,v)=\langle u,\Phi(v)\rangle$, $u,v\in M$, is an anti-linear bijective map and $\Phi(a_nu)=\phi(a)_n\Phi(u)$ for $a\in V$ and $u\in M$.
\item If there exists an anti-linear bijective map $\Phi:M\to M'$ such that $\Phi(a_nu)=\phi(a)_n\Phi(u)$ for $a\in V$ and $u\in M$, then $(u,v)=\langle u,\Phi(v)\rangle$, $u,v\in M$, is a non-degenerate invariant sesquilinear form on $M$.
\end{enumerate}
\end{lem}

The proof of the following lemma can be found in \cite{CLS}. The main point is that the product of two anti-automorphisms is an automorphism and it acts on an irreducible $V$-module as a scalar. 
\begin{lem}\label{uniqueF}
Let $(V,\phi)$ be a unitary VOA. Let $M$ be an irreducible $V$-module. 
Then there exists at most one non-degenerate invariant sesquilinear form on $M$ (up to scalar).
\end{lem}

\subsection{Unitary automorphisms and orbifold subVOAs}
Let $(V,\phi)$ be a unitary VOA and $(\ , \ )$ the corresponding positive definite invariant Hermitian form. 
We use $\Aut_{(\ , \ )}(V)$ to denote the subgroup of $\Aut(V)$ which preserves the Hermitian form, i.e, 
\[
\Aut_{(\ , \ )}(V)=\{ g\in \Aut(V)\mid  (gx,gy)=(x,y) \text{ for all } x, y\in V\}.
\]

Next lemma follows immediately from the definition (see \cite{CKLW}). 
\begin{lem} 
Let $(V,\phi)$ be a unitary VOA. Then 

(1)  $g\in \Aut_{(\ , \ )}(V)$ if and only if  $g^{-1}\phi g =\phi$. 

(2) For any $H < \Aut_{(\ , \ )}(V)$,  $(V^H, \phi)$ is also a unitary VOA. 
\end{lem}

\section{Lattice VOA}
Next we review some facts about lattice VOAs and their unitary structures. Let $L$ be a positive-definite even lattice. 
Let $V_L = M(1)\otimes_\C \C\{L\}$ be the lattice VOA as defined in \cite{FLM}. 
Let $L^*$ be the dual lattice of $L$. 
Then $V_{L^*}= M(1)\otimes_\C \C \{L^*\}$ 
is a $V_L$-module and for any coset $\lambda+L\in L^*/L$,
$V_{\lambda+L}=M(1)\otimes_\C\C\{\lambda+L\}$ is an irreducible $V_L$-module.  
It is proved in \cite{KR} that there is a positive-definite Hermitian form on $M(1)=\Span_{\C}\{\alpha_1(-n_1)\dots\alpha_k(-n_k)\1\mid \alpha_i\in L,\ n_i\in\Z_{>0}\}$ such that $(\1,\1)=1$, $(\alpha(n) u, v)= (u, \alpha(-n) v)$ for $\alpha \in L$ and for any $u,v\in M(1)$.   There also exists a positive-definite Hermitian form on $\C\{L^*\}=\Span_{\C}\{e^\alpha\mid \alpha\in L^*\}$ determined by $(e^\alpha, e^\beta)=\delta_{\alpha, \beta}$.  
Then a positive-definite Hermitian form on $V_{L^*}$ can be defined by $$
(u\otimes e^\alpha, v\otimes e^\beta)= (u,v)\cdot (e^\alpha, e^\beta),$$ 
where $u,v\in M(1)$ and $\alpha,\beta\in L$.

Let $\phi: V_{L^*}\to V_{L^*}$ be an anti-linear map determined by:
\[
\alpha_1(-n_1)\cdots \alpha_k(-n_k)\otimes e^\alpha \mapsto (-1)^k \alpha_1(-n_1)\cdots \alpha_k(-n_k)\otimes e^{-\alpha},
\]
where $\alpha_1, \dots, \alpha_k\in L, \alpha\in L^*$. 

\begin{thm}{\rm (\cite[Theorem 4.12]{Dlin})} Let $L$ be a positive-definite even lattice and let $\phi$ be the anti-linear
	map of $V_L$ defined as above. Then the lattice vertex operator algebra $(V_L , \phi)$ is
	a unitary vertex operator algebra. Moreover,  $V_{\lambda+L}$ is a unitary module of $V_L$ for each $\lambda+L\in L^*/L$. 
\end{thm}

\subsection{$\Aut_{(\,, \,)}(V_L)$}
Next we consider some automorphisms of $V_L$ which preserves the invariant Hermitian form. 
First we recall some facts about the automorphism group of $V_L$. Let $L$ be an even lattice with the (positive-definite) bilinear form $\langle \cdot   | \cdot\rangle$.
Denote  by $\hat{L}=\{\pm e^\alpha \mid \alpha \in L\}$ a 
central extension of $L$ by $\pm 1$
such that $e^{\alpha} e^\beta = (-1)^{(\alpha |\beta)} e^{\beta}e^\alpha $.
Let $\Aut(\hat{L})$ be the automorphism group of $\hat{L}$ as a group. We also assume that 
$e^\alpha \cdot e^{-\alpha}=(-1)^{\langle \alpha|\alpha\rangle /2} e^0$. 
For $g\in\Aut(\hat{L})$, let $\bar{g}$ be the map $L\to L$ defined by $g(e^\alpha)\in\{\pm e^{\bar{g}(\alpha)}\}$.
Let $O(\hat{L})=\{g\in\Aut(\hat{L})\mid \bar{g}\in O(L)\}$.
Then by \cite[Proposition 5.4.1]{FLM},  we have an exact sequence
\begin{equation}\label{iota}
  1 \to \Hom (L,\Z/2\Z) \to  O(\hat{L})\xrightarrow{\iota}  O(L) \to  1.
\end{equation} 
It is known that $O(\hat{L})$ is a subgroup of $\Aut(V_L)$ (cf.~loc.~cit.).
Let
$$
  \Inn(V_L) = \left\langle \exp(a_{(0)}) \mid a\in (V_L)_1 \right\rangle
$$
be the normal subgroup of $\Aut(V_L)$ generated by the inner automorphisms $\exp(a_{(0)})$.

\begin{thm}[\cite{DN}]\label{aut}
Let $L$ be a positive definite even lattice.
Then
\[
  \Aut (V_L) = \Inn(V_L)\,O(\hat{L})
\]
Moreover, the intersection $\Inn(V_L)\cap O(\hat{L})$ contains a subgroup
$\hom (L,\Z/2\Z)$ and the quotient $\Aut (V_L)/\Inn(V_L)$ is isomorphic
to a quotient group of $O(L)$.
\end{thm}

The following  lemmas can be proved easily from the definition.  
\begin{lem}
Let $g\in O(\hat{L})$. Then $g\in \Aut_{(\ , \ )}(V_L)$.
\end{lem}

\begin{proof}
Let $g\in O(\hat{L})$. Set $g(e^{\alpha})=a_\alpha e^{\bar{g}\alpha}$ and  $g(e^{-\alpha})=b_\alpha e^{-\bar{g}\alpha}$ for some roots of unity $a_\alpha, b_\alpha\in \C$.  Recall that $e^{0}$ is the identity of $\hat{L}$ and  $e^{\alpha}\cdot e^{-\alpha}= (-1)^{\langle \alpha, \alpha\rangle/2} e^{0}$.  Then 
\[
(-1)^{\langle \alpha, \alpha\rangle/2} e^{0}= g(e^{\alpha}\cdot e^{-\alpha})= a_\alpha b_\alpha
e^{\bar{g}\alpha}\cdot  e^{-\bar{g}\alpha}= a_\alpha b_\alpha (-1)^{\langle \bar{g}\alpha, \bar{g}\alpha\rangle/2} e^{0}  
\]
and we have $a_\alpha b_\alpha=1$.   
Then 
\[
g \phi(\alpha_1(-n_1)\cdots \alpha_k(-n_k)\otimes e^\alpha)
= (-1)^k b_\alpha \bar{g}\alpha_1(-n_1)\cdots \bar{g} \alpha_k(-n_k)\otimes e^{-\bar{g}\alpha}
\]
and 
\[
 \phi g(\alpha_1(-n_1)\cdots \alpha_k(-n_k)\otimes e^\alpha)
= (-1)^k \overline{a_\alpha} \bar{g}\alpha_1(-n_1)\cdots \bar{g} \alpha_k(-n_k)\otimes e^{-\bar{g}\alpha}.
\]
Since $b_\alpha= \overline{a_\alpha}$, we have $g\phi=\phi g$ as desired. 
\end{proof}

\begin{lem}
Let $\beta \in L^*$ and $n$ a positive integer. Then $h= \exp(2\pi i \frac{\beta(0)}n) \in \Aut_{(\ , \ )}(V_L)$. 
\end{lem}

\begin{proof}
	Let $h= \exp(2\pi i \frac{\beta(0)}n)$. Then 
\[
\begin{split}
h \phi(\alpha_1(-n_1)\cdots \alpha_k(-n_k)\otimes e^\alpha)
= &(-1)^k h( \alpha_1(-n_1)\cdots \alpha_k(-n_k)\otimes e^{-\alpha})\\
= & (-1)^k \exp(-2\pi i \langle \beta| \alpha\rangle/n ) \alpha_1(-n_1)\cdots \alpha_k(-n_k)\otimes e^{-\alpha}\\
= & \phi (\exp(2\pi i \langle \beta| \alpha\rangle/n ) \alpha_1(-n_1)\cdots \alpha_k(-n_k)\otimes e^{\alpha})\\
=&\phi h (\alpha_1(-n_1)\cdots \alpha_k(-n_k)\otimes e^\alpha)
\end{split}
\]
as desired. Note that $\phi$ is an anti-linear map. 
 \end{proof}

\subsection{Unitary form on twisted modules} \label{sec:5.2}
Next we discuss a unitary form on a twisted module.  The main idea is similar to that in \cite{CLS,Dlin}. First we review the construction of twisted $V_L$-modules from \cite{Dl96} and \cite{Lep85}.

Let $L$ be an even positive-definite lattice with a $\Z$-bilinear form $\langle \cdot|\cdot \rangle$. Let $\tau$ be an isometry  of  $L$. Let $p$ be a positive integer such that $\tau^p=1$ but $p$ may not be the order of $\tau$. 
Define $\mathfrak{h}=\C\otimes_{\Z}L$ and
extend the $\Z$-form $\langle \cdot|\cdot \rangle$  $\C$-linearly to $\h .$
Denote
$$\h_{(n)}=\{ \alpha\in
\h\,|\, \tau\alpha= \xi^n \alpha \} \quad \text{for } n\in \Z, $$ where
$\xi=\exp({2\pi \sqrt{-1}/p})$.
In particular,  $\h_{(0)} =\h^\tau$ is the fixed point subspace of $\tau$ on $\h$. 

Let $\hat{\h}[\tau]=\coprod_{n\in \Z}\h_{(n)}\otimes t^{n/p}\oplus \C c$ be 
the  $\tau$-twisted affine Lie algebra of $\h$. 
Denote
\[
\hat{\h}[\tau]^+=\coprod_{n>0}\h_{(n)}\otimes t^{n/p},\quad
\hat{\h}[\tau]^-=\coprod_{n<0}\h_{(n)}\otimes t^{n/p},\quad
\text{and} \quad \hat{\h}[\tau]^0 =\h_{(0)}\otimes t^{0}\oplus \C c,
\]
and form an induced module
\[
S[\tau]=U(\hat{\h}[\tau])\otimes_{U(\hat{\h}[\tau]^+\oplus \hat{\h}[\tau]^0)}
\C \cong S(\hat{\h}[\tau]^-) \quad \text{(linearly),}
\]
where $\coprod_{n>0}\h_{(n)}\otimes t^{n/p}$ acts trivially on
$\C$ and $c$ acts as $1$, and $U(\cdot)$ and $S(\cdot)$ denote the
universal enveloping algebra and symmetric algebra, respectively. For any $\alpha\in L$ and $n\in \frac{1}p \Z$, let $\alpha_{(pn)}$ be the natural projection of $\alpha$ in  $\h_{(pn)}$ and we denote $\alpha(n) = \alpha_{(pn)}\otimes t^{n}$.

Set $s=p$ if $p$ is even and $s=2p$ if $p$ is odd.  
Following \cite[Remark 2.2]{Dl96}, 
we define a $\tau$-invariant alternating $\Z$-bilinear map 
$c^{\tau}$ from $L \times L$ to $\Z_s$ by 
\begin{equation}\label{eq:c_cs}
c^\tau(\alpha,\beta) 
= \sum_{i=0}^{p-1} (s/2+si/p) \langle  \tau^i (\alpha)|\beta \rangle + s\Z.
\end{equation}

For any positive integer $n$, let $\la \kappa_n\ra $ be a cyclic group of order $n$ and consider the central extension
\[
1\ \longrightarrow\  \la \kappa_s\ra \ \longrightarrow\  \hat{L}_\tau \
\bar{\longrightarrow\ } L \longrightarrow\  1
\]
such that
$ aba^{-1}b^{-1}=\kappa_s^{c^\tau(\bar{a},\bar{b})}$ 
for  $a,b\in \hat{L}_\tau $. 
Recall that there is a set-theoretic identification between the central extensions $\hat{L}$ and $\hat{L}_\tau$ such
that the respective group multiplications $\times$ and $\times_\tau$
are related by
\begin{equation}\label{Rem:Setid}
a\times b= \kappa_s^{\varepsilon_0(\bar{a}, \bar{b})} a\times_\tau b,
\end{equation}
where $\displaystyle \varepsilon_0(\alpha, \beta)= \sum_{0<r<p/2} (s/2+rs/p)\langle \tau^{-r} \alpha| \beta\rangle$ (see \cite[Remark 2.1]{Dl96}).  

Now let $\hat{\tau}$ be a standard lift of $\tau$ in $O(\hat{L})$, i.e., $\hat{\tau}(e^\alpha)= e^\alpha $ for any $\alpha \in L^\tau$.  Then $\hat{\tau}$  is also an automorphism of $\hat{L}_\tau$ by the identification given in \eqref{Rem:Setid}. 

Next we recall a construction of an irreducible $\hat{L}_\tau$-module on which  $K=\{a^{-1}\hat{\tau}(a)\mid a\in
\hat{L}_\tau\}$ acts trivially and $\kappa_s$ acts as
multiplication by $\xi=\exp(2\pi\sqrt{-1}/s)$ (cf. \cite[Proposition 6.1]{Lep85} ). 
Let $P_0: \h \to h_{(0)}$ be the natural projection. Set $$N= (1-P_0)\h \cap L=\{ \alpha\in L \mid  \langle \alpha| \h_{(0)} \rangle=0\},$$ 
$R=\{\alpha\in N\,|\,c^\tau(\alpha,\beta)=0\text{ for }\beta\in N\}$ and $M= (1-P_0)L$. 
Denoting by $\widehat{Q}_\tau$ the subgroup of $\widehat{L}_\tau$ obtained by pulling back a subgroup $Q$  of $L$. 
 Then $\widehat{R}_\tau$  is the center of $\widehat{N}_\tau$ and that $\widehat{M}_\tau  \subset \widehat{R}_\tau$. Note 
also that $K=\{a^{-1}\hat{\tau}(a)\mid a\in
\hat{L}_\tau\}< \widehat{M}_\tau  < \widehat{R}_\tau$.  

Let $\mathcal{A}>\widehat{R}_\tau> K$ be a maximal abelian subgroup of $\hat{N}_\tau$.  
Let $\chi: \mathcal{A}/K\to \C$ be a linear character of $\mathcal{A}/K$. Let $\C_\chi$ be the $1$-dimensional module  of $\mathcal{A}$ affording $\chi$. Then we obtain an irreducible  $\widehat{N}_\tau$-module $T_\chi$ and an irreducible $\hat{L}_\tau$-module $U_\chi$ as follows: 
\[
T_\chi= \C[\hat{N}_\tau]\otimes_{\C[\mathcal{A}]} \C_\chi \quad \text{ and } \quad U_\chi= \C[\hat{L}_\tau]\otimes_{\C[\mathcal{A}]} \C_\chi = \C[P_0(L) ]\otimes T_\chi.
\]
The twisted space $V_L^\chi(\tau )=S[\tau ]\otimes U_\chi$
forms an irreducible $\hat{\tau}$-twisted $V_L$-module with the vertex operator $Y^\tau(\cdot, z): V_L \to
\mathrm{End}(V_L^T) [[z,z^{-1}]]$ on $V_L^T$ defined as follows (cf. \cite{Dl96}): 
For
$a\in\hat{L}$, define
\begin{equation}\label{Wtau}
W^\tau(a,z)=p^{-\la\bar{a}|\bar{a}\ra/2} \sigma(\bar{a})
E^-(-\bar{a},z)E^{+}(-\bar{a},z) a z^{-\la\bar{a}|\bar{a}\ra/2},
\end{equation}
where
\[ 
E^{\pm}(\alpha,z)=\exp\left(\sum_{n\in \frac{1}p\mathbb{Z}^{\pm}}
\frac{\alpha(n)}{n}z^{-n}\right)
\]
and 
\begin{equation}\label{sigma}
\sigma(\alpha)= 
\begin{cases}
\displaystyle \prod_{0<r< p/2} (1-{\xi}^{-r})^{\langle \tau^r \alpha|\alpha\rangle} 
2^{\langle\tau^{p/2}\alpha|\alpha\rangle} & \text{ if } p\in 2\Z,\\
\displaystyle\prod_{0<r< p/2} (1-{\xi}^{-r})^{\langle \tau^r \alpha|\alpha\rangle} 
& \text{ if } p\in 2\Z+1.
\end{cases}
\end{equation}
Note that $a\in\hat{L}$ acts on $U_\chi$ as an element of $\hat{L}_\tau$ via the identification given in \eqref{Rem:Setid}.

For $\alpha_1,\dots,\alpha_k \in\mathfrak{ h}$, $n_1,\dots,n_k>0$,
and $v=\alpha_1(-n_1)\cdots\alpha_k(-n_k)\cdot\iota(a)\in V_L,$
set
\begin{equation*}
W(v,z)=\nop 
\left(\frac{1}{(n_1-1)!}\left(\frac{d}{dz}\right)^{n_1-1}
\alpha_1(z)\right)\cdot\cdot\cdot\left(\frac{1}{(n_k-1)!}
\left(\frac{d}{dz}\right)^{n_k-1} \alpha_k(z)\right)W^{\tau}(a,z)
\nop,
\end{equation*}
where $\alpha(z)=\sum_{n\in \Z/p} \alpha(n)z^{-n-1}$ and $\nop\cdots \nop$ denotes the normal ordered product.

Define constants $c_{mn}^i\in\mathbb{ C}$ for $m, n\ge 0$ and
$i=0,\cdots, p-1$ by the formulas
\begin{gather}\label{cmn}
\sum_{m,n\ge 0}c_{mn}^0x^my^n=-\frac{1}{2}\sum_{r=1}^{p-1}{\rm
log}\left(\frac
{(1+x)^{1/p}-\xi^{-r}(1+y)^{1/p}}{1-\xi^{-r}}\right),\\
\sum_{m,n\ge 0}c_{mn}^ix^my^n=\frac{1}{2}{\rm log}\left( \frac
{(1+x)^{1/p}-\xi^{-i}(1+y)^{1/p}}{1-\xi^{-i}}\right)\ \text{
for}\ \ i\ne0.
\end{gather}

Let $\{\beta_1,\cdot\cdot \cdot, \beta_d\}$ be an orthonormal
basis of $\mathfrak{h}$ and set
\begin{equation}
\Delta_z=\sum_{m,n\ge 0}\displaystyle{\sum^{p-1}_{i=0}}\
\displaystyle{ \sum^d_{j=1}}
c_{mn}^i(\tau^{-i}\beta_j)(m)\beta_j(n)z^{-m-n}.
\end{equation}
Then $e^{\Delta_z}$ is well-defined on $V_L$ since $c_{00}^i=0$
for all $i$, and for $v\in V_L,$ $e^{\Delta_z}v\in V_L[z^{-1}].$
Note that $\Delta_z$ is independent of the choice of orthonormal
basis and
\begin{equation*} \hat{\tau}\Delta_z=\Delta_z\hat{\tau} \qquad
\text{and} \qquad \hat{\tau}e^{\Delta_z}=e^{\Delta_z}\hat{\tau}\quad
\text{ on } V_L.
\end{equation*}
For $v\in V_L,$  the vertex operator $Y^{\tau}(v,z)$ is defined by
\begin{equation}\label{dyg}
Y^{\tau}(v,z)=W(e^{\Delta_z}v,z).
\end{equation}

Let $\beta \in \mathbb{Q}\otimes L^\tau$ such that $p\langle \beta|L\rangle \in \Z$. Then 
$g=\hat{\tau}\exp(2\pi i \beta(0))$ also defines an automorphism of $V_L$  and $g^p=1$. An  irreducible $g$-twisted module  is then given by 
$$V_L^{\chi}(g) =S[\tau]\otimes e^{-\beta}\otimes  U_\chi \cong S[\tau]\otimes\C[P_0^\tau(L)-\beta]\otimes T_\chi,$$
as a vector space. The vertex operator is still given by  $Y^{\tau}(v,z)$ but the action of $a\in \hat{L}$ on $U_\chi$  is twisted by $e^{-\beta}$.   Note that the alternating map $c^\tau(\cdot ,\cdot )$ is still well-defined on $\tilde{L}=\mathrm{Span}_\Z\{L, \beta\}$ and 
$a\cdot (e^{-\beta}\otimes u)= \xi^{-\langle \bar{a}|  \beta \rangle}   e^{-\beta}\otimes a\cdot u$
for any $a\in \hat{L}$ and $u\in U_\chi$.


Next we define a Hermitian form on $V_L^{\chi}(g)$ as follows.  For any $a, b\in e^{-\beta}\hat{L}_\tau$, define 
\begin{equation}\label{ta}
(t(a), t(b)) = 
\begin{cases}
0 & \text{ if }  b^{-1}a\not\in \mathcal{A} , \\
\chi(b^{-1}a) & \text{ if }  b^{-1}a\in \mathcal{A},
\end{cases}
\end{equation}
where $t(a) = a \otimes 1 \in e^{-\beta}\otimes U_\chi$. Using the similar arguments as in \cite{FLM,KR}, there
is a positive-definite Hermitian form $(\ ,\ )$ on $S[\tau]$ such that
\[
\begin{split}
(1, 1) &= 1,\\
(\alpha(n) \cdot  u, v) & = (u, \alpha(-n)\cdot v),
\end{split}
\]
for any $u, v \in S[\tau]$ and $\alpha\in L$. Then one can  define a positive-definite
Hermitian form on $V_L^{\chi}(g)$ by 
\[
(u\otimes r, v\otimes s)= (u,v)\cdot (r,s), \quad \text{ where } u,v\in S[\tau], r,s\in e^{-\beta}\otimes U_\chi.
\] 

\begin{lem}\label{ealpha}
For any $u,v\in V_L^{\chi}(g)$ and $a\in \hat{L}$, we have $(a\cdot u, a\cdot v)= (u,v)$.
\end{lem}

\begin{proof}
It suffices to consider the case for 
\[
u=v_1\otimes t(b_1)\quad \text{ and } \quad v=v_2\otimes t(b_2), 
\]
where 
$ v_1, v_2\in S[\tau]$ and $b_1,b_2\in e^{-\beta}\hat{L}_\tau$.
By definition, we have $$(a\cdot u, a\cdot v)= (v_1\otimes t(ab_1), v_2\otimes t(ab_2))=
(v_1, v_2)\cdot (t(ab_1), t(ab_2)).$$
Moreover, $(ab_2)^{-1} a b_1 = b_2^{-1} a^{-1} a b_1=b_2^{-1}b_1$. Therefore, we have 
$\chi((ab_2)^{-1}ab_1) =  \chi(b_2^{-1}b_1)$
if $b_2^{-1}b_1\in \mathcal{A}$. Hence, we  have 
$(a\cdot u, a\cdot v)= (u,v)$ as desired. 
\end{proof}

\begin{lem}
For any $\alpha\in L$ and $u,v\in V_L^{\chi}(g)$, we have 
\[
(e^\alpha \cdot u, v)= (u, \mu e^{-\alpha}\cdot v)
\]
where 
\[
\mu = \begin{cases}
	\xi^{-\sum_{0<r<p/2} r\langle \tau^r\alpha, \alpha\rangle} &\text{ if  $p$ is odd},\\
	\xi^{-\sum_{0<r<p/2} r\langle \tau^r\alpha, \alpha\rangle} (-1)^{ \frac{1}2 \langle \tau^{p/2} \alpha, \alpha\rangle} &\text{ if  $p$ is even},\\
\end{cases} 
\] 
\end{lem}

\begin{proof}
Recall the set-theoretic identification between $\hat{L}$ and $\hat{L}_\tau$ given in \eqref{Rem:Setid}. 
It follows from $e^{\alpha} \times e^{-\alpha} =\kappa_2^{\la \alpha| \alpha\ra /2} e^0$ that 
\[
\begin{split}
	e^{\alpha}\times_\tau e^{-\alpha} &=  \kappa_s^{\varepsilon_0(\alpha,\alpha)}\kappa_2^{\langle \alpha| \alpha\rangle/2 }e^0, \\
	&= 
	\begin{cases}
		\kappa_p^{\sum_{0<r<p/2} r\langle \tau^r\alpha| \alpha\rangle} e^0 &\text{ if  $p$ is odd},  \\
		\kappa_p^{\sum_{0<r<p/2} r\langle \tau^r\alpha| \alpha\rangle} \kappa_2^{ \frac{1}2 \langle \tau^{p/2} \alpha| \alpha\rangle} e^0  &\text{ if  $p$ is even}. 
	\end{cases}
\end{split}
\] 
Now by Lemma \ref{ealpha}, we have 
\[
(u, \mu e^{-\alpha}v)=  (e^\alpha\cdot u, e^\alpha\cdot(\mu  e^{-\alpha}v) ) =(e^\alpha \cdot u, v)
\]
as desired. 
\end{proof}

\begin{lem}
	For $\alpha\in L$, $\overline{\sigma(\alpha)} \mu = (-1)^{\langle \alpha| \alpha\rangle/2} \sigma(\alpha)$. 
\end{lem}

\begin{proof}
	By definition (see \eqref{sigma}), we have 
	\[
	\begin{split}
	\overline{\sigma(\alpha)} \mu = &   \displaystyle\prod_{0<r< p/2} \overline {(1-{\xi}^{-r})^{\langle \tau^r \alpha|\alpha\rangle}} \cdot \xi^{-\sum_{0<r<p/2} r\langle \tau^r\alpha, \alpha\rangle}\\
	 = & (-1)^{\langle \alpha, \alpha\rangle/2} \prod_{0<r< p/2} (1-{\xi}^{-r})^{\langle \tau^r \alpha|\alpha\rangle} = (-1)^{\langle \alpha| \alpha\rangle/2} \sigma(\alpha)
	\end{split}
	\]
	if $p$ is odd and 
\[
\begin{split}
\overline{\sigma(\alpha)} \mu = &   \displaystyle\prod_{0<r< p/2} \overline {(1-{\xi}^{-r})^{\langle \tau^r \alpha|\alpha\rangle} 2^{\langle \tau^{p/2}\alpha| \alpha\rangle } }\cdot \xi^{-\sum_{0<r<p/2} r\langle \tau^r\alpha, \alpha\rangle} \cdot (-1)^{\langle \tau^{p/2} \alpha| \alpha\rangle /2}\\ 
= &(-1)^{\langle \alpha| \alpha\rangle/2} \prod_{0<r< p/2} (1-{\xi}^{-r})^{\langle \tau^r \alpha|\alpha\rangle} 2^{\langle \tau^{p/2}\alpha| \alpha\rangle } =(-1)^{\langle \alpha| \alpha \rangle/2} \sigma(\alpha)
\end{split}
\]	
if $p$ is even. 
\end{proof}	

The proof of the following lemma is very similar to that in \cite[Theorem 4.14]{Dlin} and \cite[Lemma 5.6]{CLS}. 

\begin{lem}\label{twisted}
For any $\chi$, $V_L^{\chi}(g)$ is a unitary $g$-twisted module of $(V_L,\phi)$. 
\end{lem}

\begin{proof}
We only need to verify the invariant property. 
Since the VOA $V_L$ is generated by $\{\alpha(-1)\cdot \1\mid \alpha\in L\} \cup \{e^\alpha \mid \alpha\in L\}$, 
it is sufficient to check
\[
(Y^\tau (e^{zL(1)} (-z^{-2} )^{L(0)} x, z^{-1})u, v) = (u, Y^\tau (\phi(x), z)v)
\]
for $x\in \{\alpha(-1)\cdot \1\mid \alpha\in L\} \cup \{e^\alpha \mid \alpha\in L\}$ and  $u,v \in V_L^{\chi}(g)$ (cf.\ \cite[Proposition 2.12]{Dlin}).

Let  $u = v_1\otimes t(a)$ and $v = v_2 \otimes t(b)$ for some $v_1 , v_2 \in S[\tau]$, $a, b \in e^{-\beta}\hat{L}_\tau$. Then
\[
(\alpha(n) u, v) = (u, \alpha(-n)v)
\]
for any $\alpha\in L$ and $n \in \frac{1}p\Z$. Thus for $x = \alpha(-1)\cdot \1$, we have
\[
\begin{split}
& (Y^\tau (e^{zL(1)} (-z^{-2} )^{ L(0)}\alpha (-1)\cdot \1 , z^{-1}) u, v)\\
= &-z^{-2} (Y^\tau(\alpha(-1)\cdot \1, z^{-1}) v_1 \otimes t(a), v_2\otimes t(b))\\
=& -z^{-2}\sum_{n\in \frac{1}p\Z} (\alpha(n)v_1 , v_2 )(t(a), t(b))z^{n+1}\\
= & -\sum_{n\in \frac{1}p\Z} (v_1 , \alpha(-n)v_2 )(t(a), t(b)) z^{n-1}\\
= &(u, Y^\tau (\phi(\alpha(-1)\cdot \1), z)v).
\end{split}
\]
Notice that $e^{\Delta_z}(\alpha(-1)\cdot \1) = \alpha(-1)\cdot \1$.

Now take $x = e^\alpha$ with $\langle \alpha| \alpha\rangle = 2k$. Then we have
\[
\begin{split}
& (Y^{\tau} (e^{ zL(1)} (-z^{-2} )^{L(0)} e^\alpha , z^{-1} )u, v)\\
= &(Y^{\tau}(e^{ zL(1)} (-z^{-2} )^{L(0)} e^\alpha , z^{-1} )v_1 \otimes t(a), v_2\otimes t(b))\\
= &(-z^{-2} )^k (p^{-k} \sigma(\alpha) E^-(-\alpha,z^{-1})E^{+}(-\alpha,z^{-1}) e^\alpha z^{k} v_1 \otimes t(a), v_2\otimes t(b))\\
= & (-z^{-2} )^k  (  v_1 \otimes t(a), p^{-k} \overline{ \sigma(\alpha)}E^-(\alpha,z)E^{+}(\alpha,z) \mu e^{-\alpha} z^{k}v_2\otimes t(b))\\
= & (  v_1 \otimes t(a), p^{-k} \sigma(\alpha) E^-(\alpha,z)E^{+}(\alpha,z) e^{-\alpha} z^{-k}v_2\otimes t(b))\\
=& (  v_1 \otimes t(a), Y^\tau(\phi(e^\alpha), z) v_2\otimes t(b)) 
\end{split}
\]
as desired.  
\end{proof}

\section{Holomorphic VOAs of central charge $24$}
In this section, we review a construction of holomorphic vertex operator algebras of central charge $24$ using certain simple current extensions of  lattice vertex operator algebras and some orbifold vertex operator subalgebras in the Leech lattice vertex operator algebra  \cite{Ho2,La19,BLS}.

Assume that $V_1$ is semisimple and let $\mathfrak{h}$ be a Cartan subalgebra of $V_1$. Let $M(\mathfrak{h})$ be the subVOA generated by $\mathfrak{h}$ and denote 
\[
W=\mathrm{Comm}_V(M(\mathfrak{h})) \quad \text{ and } X =\mathrm{Comm}_V(W). 
\]
Then $X$ is isomorphic to a lattice VOA $V_L$ and  $W$ is isomorphic to an orbifold VOA
$V_{\Lambda_\tau}^{\hat{\tau}}$, where  $\Lambda_\tau$ is the coinvariant sublattice of the Leech lattice $\Lambda$ associated with an isometry $\tau\in O(\Lambda)$  (see \cite{Ho2} and \cite{La19}).  The possible isometry $\tau\in O(\Lambda)$ has been described in \cite{Ho2}. 
It is also proved in \cite{La19} that all irreducible modules for the fixed point subVOA 
$V_{\Lambda_\tau}^{\hat{\tau}}$ are simple current modules. Therefore, the VOA $V$ can be viewed as a simple current extension of $V_{L}\otimes V_{\Lambda_\tau}^{\hat{\tau}}$.

\subsection{Automorphism groups and $\mathrm{Stab}_{\Aut(V)}(V_{L}\otimes V_{\Lambda_\tau}^{\hat{\tau}})$}

Next we describe the subgroup $\mathrm{Stab}_{\Aut(V)}(V_{L}\otimes V_{\Lambda_\tau}^{\hat{\tau}})$ for each case  by using the methods in \cite{Sh07}. Note that the automorphism groups for all holomorphic VOAs of central charge $24$  with $V_1\neq 0$ and $\mathrm{Stab}_{\Aut(V)}(V_{L}\otimes V_{\Lambda_\tau}^{\hat{\tau}})$ have already  been computed in \cite{BLS}. 

\begin{nota}\label{gconjandMuU}
For any VOA $U$, $\Aut(U)$ acts on $\mathrm{Irr}(U)$ by module conjugations: 
for a $V$-module $\left(M,Y_{M}\right)$ and $g\in \Aut(V)$, the $g$-conjugate module 
$\left(g\circ M,Y_{g\circ M}\right)$ of $(M, Y_M)$  is defined  by $g\circ M = M$
as a vector space and $Y_{g \circ M}\left(v,z\right)=Y_{M}\left(g^{-1}v,z\right)$
for $v\in V$. This action preserves the conformal weights. Thus, we have a canonical group homomorphism 
\begin{equation}
	\mu_U:\Aut(U)\to O(\irr(U),q_U),\label{Eq:Orthogonal}
\end{equation}
where $O(\irr(U),q_U)=\{h\in\Aut(\irr(U))\mid q_U(M)=q_U(h(M))\ \text{for all}\ M\in\irr(U)\}$  is the orthogonal group of the quadratic space $(\irr(U),q_U)$.  
We use $\overline{\Aut}(U)$ and  $\Aut_0(U)$ to denote $\Imm\mu_{U}$ and $\Ker\mu_U$, respectively. 
\end{nota}

Recall that the irreducible modules for the lattice VOA $V_L$ are parametrized by its discriminant group $\mathcal{D}(L)=L^*/L$ and the fusion rules are given by 
$V_{\lambda+L} \times V_{\eta+L}=V_{\lambda+\eta+L}$ for $\lambda, \eta\in L^*$. 
Since $O(L)$ acts naturally on $\mathcal{D}(L)$, we also have a canonical group homomorphism  $\mu_L:O(L)\to O(\mathcal{D}(L),q_L)$. We use $\overline{O}(L)$ to denote the image of $\mu_L$ on $O(\mathcal{D}(L),q_L)$. 
 
Set $W=V_{\Lambda_\tau}^{\hat{\tau}}$ and   let $\varphi$ be a bijection from $\mathcal{D}(L)$ to $\irr(W)$ such that 
 \begin{equation*}
 	V\cong\bigoplus_{\lambda+L\in\mathcal{D}(L)}V_{\lambda+L}\otimes \varphi({\lambda+L})\label{Eq:V}.
 \end{equation*}
For simplicity, we often denote $\varphi({\lambda+L})$ by $W_\lambda$.

Set $S_\varphi=\{(V_{\lambda+L},\varphi(\lambda+L))\mid   \lambda+L\in\mathcal{D}({L})\}\subset\irr(V_L)\times\irr(W)$.  
The dual group $S_\varphi^*={\rm Hom}(S_\varphi,\C^\times)$ acts faithfully on $V$ with the action given by  
 \begin{equation}
 	S_\varphi^*=\{\exp(2\pi\sqrt{-1}v_{(0)})\mid v+L\in \mathcal{D}(L)\}.\label{Eq:S*}
 \end{equation}
 By \cite[Theorem 3.3]{Sh07}, we know that 
 \begin{align}
 	S_\varphi^*&=\{\sigma\in\Aut(V)\mid \sigma=id\ {\rm on}\ V_L\otimes W\}\label{Eq:ns0}.
 \end{align}
and 
\begin{align}
	N_{\Aut(V)}(S_\varphi^*)/S_\varphi^*&\cong \Stab_{\Aut(V_{L}\otimes W)}(S_\varphi)=\{\sigma\in \Aut(V_{L}\otimes W)\mid \sigma\circ S_\varphi  =S_\varphi\}.\label{Eq:ns}
\end{align}
Note that  
\begin{align}
	N_{\Aut(V)}(S_\varphi^*)=\{\sigma\in\Aut(V)\mid \sigma(V_{L}\otimes W)=V_{L}\otimes W\}= 	\Stab_{\Aut(V)}(V_{L}\otimes W).\label{Eq:NS*}
\end{align}
Set ${\rm Stab}_{\Aut(V)}(\h)=\{\sigma\in\Aut(V)\mid \sigma(\h)=\h\}$ and ${\rm Stab}_{\Inn(V)}(\h)={\rm Stab}_{\Aut(V)}(\h)\cap\Inn(V), $
where $\h$ is the chosen Cartan subalgebra of $V_1$.  By \cite[Lemma 3.14]{BLS}, 
\[
\Aut(V) = \Inn(V){\rm Stab}_{\Aut(V)}(\h) \quad \text{ and } \quad N_{\Aut(V)}(S^*_\varphi)=\Inn(V_L){\rm Stab}_{\Aut(V)}(\h). 
\]
Moreover,  $\Stab_{\Aut(V)}(\h)/S_\varphi^* \cong \Stab_{\Aut(V_{L}\otimes W)}(S_\varphi)\cap\Stab_{\Aut(V_{L}\otimes W)}(\h)$ \cite[Lemma 3.15]{BLS}. 

Recall from \cite[Theorem 3.4]{BLS} that $\mu_{W}$ is injective and $\overline{\Aut}(W)\cong \Aut(W)$.  Therefore, the kernel of the group homomorphism 
\[
\Aut(V_L\otimes W)\to  O(\irr(V_L),q_{V_L})\times O(\irr(W),-q_W),\quad \sigma\mapsto (\mu_{V_L}(\sigma_{|V_L}),{\mu}_W(\sigma_{|W}))
\]
is $\Aut_0({V_L})\times 1$.  It turns out that $\Stab_{\Aut(V_{L}\otimes W)}(S_\varphi)$ may be viewed as a subgroup of $\Aut(V_L)$ by considering the restriction of $\Stab_{\Aut(V_{L}\otimes W)}(S_\varphi)$ to $V_L$. We also have   
\begin{equation}
	\Stab_{\Aut(V_{L}\otimes W)}(S_\varphi)\cong \Aut_0({V_L}).(\overline{O}(L)\cap \varphi^{*}(\overline{\Aut}(W)))< \Aut(V_L),\label{Eq:StabS}
\end{equation}
where  $\varphi^*(\overline{\Aut}(W))=\varphi^{-1}(\overline{\Aut}(W))\varphi\subset O(\mathcal{D}(L),q_L)$ and  
\begin{align*}
	&\Stab_{\Aut(V_{L}\otimes W)}(S_\varphi)\cap\Stab_{\Aut(V_{L}\otimes W)}(\h)\\\cong& \{\exp(a_{(0)})\mid a\in \h\}\iota^{-1}(O_0(L).(\overline{O}({L})\cap \varphi^{*}(\overline{\Aut}(W)))) .
\end{align*}

Let $W(V_1)$ be the Weyl group of the semisimple Lie algebra $V_1$. Since $V_1$ is a semisimple, ${\rm Stab}_{\Inn(V)}(\h)$ acts on $\h$ as $W(V_1)$. 

\begin{lem}[{\cite[Lemma 3.16]{BLS}}]\label{Lem:4.1}
	\begin{enumerate}[{\rm (1)}]
		\item ${\rm Stab}_{\Inn(V)}(\h)/\{\exp(a_{(0)})\mid a\in \h\}\cong {W}(V_1)$.
		\item  ${\rm Stab}_{\Aut(V)}(\h)/\{\exp(a_{(0)})\mid a\in \h\}\cong \mu_L^{-1}( \bar{O}(L) \cap \varphi^{*}( \overline{\Aut}(W))$.
	\end{enumerate}
\end{lem}
Therefore, we may regard $W(V_1)$ as a subgroup of $O(L)$.

\subsection{Orbifold construction from Niemeier lattice VOAs}  \label{NieOrb}

It is known that all holomorphic VOA of central charge $24$ can be constructed from a single orbifold construction from a lattice VOA.  The constructions from Leech lattice VOA have been discussed in \cite{ELMS} (see also \cite{CLM}).  The constructions from Niemeier lattice VOAs are also discussed in \cite{HM}. In particular, the following has been proved. 

\begin{thm}[cf. {\cite[Proposition 5.7 and Remark 5.8]{HM}}]\label{ThmHM}
Let $V$ be a holomorphic VOA of central charge $24$ with $V_1\neq 0$. Then there exist a Niemeier $N$ and an automorphism $g=\hat{\tau}\exp(2\pi i \beta(0))\in Aut(V_N)$ such that   $V\cong \widetilde{V_N}(g)$. Moreover, 
\begin{enumerate}
\item $\tau$ has the same frame shape as one of the $11$ conjugacy classes of $Co_0$ as discussed  in \cite{Ho2} and $|g|=|\tau|$.  

\item $L\cong N^\tau_\beta$ and $V_{N_\tau}^{\hat{\tau}} \cong V_{\Lambda_\tau}^{\hat{\tau}}$, where $N^\tau_\beta= \{ x\in N^\tau\mid  \langle x, \beta\rangle \in \Z\}$; 
in particular, $V_N^g > V_{L}\otimes V_{\Lambda_\tau}^{\hat{\tau}}$. 

\item $(V_N^g)_1$ is non-abelian and has the same Lie rank as $V_1$. 

\end{enumerate}
  
\end{thm} 

We note that the choices for $N$ and $g$ are not unique. It turns out that it is possible to choose $(N,g)$ so that $(V_N^g)_1$ contains a simple Lie ideal which is a proper Lie subalgebra of a simple ideal of $V_1$. 

\begin{prop}\label{prop4.4}
Let $V$ be a holomorphic VOA of central charge $24$ with $V_1\neq 0$ and $\rank V_1<24$. Then there exist a Niemeier $N$ and an automorphism $g=\hat{\tau}\exp(2\pi i \beta(0))\in \Aut(V_N)$ such that   $V\cong \widetilde{V_N}(g)$ and Conditions (1), (2), (3) in Theorem \ref{ThmHM} are satisfied. Moreover,  $(V_N^g)_1$ contains a simple Lie ideal which is a proper Lie subalgebra of a simple ideal of $V_1$. 

\end{prop}

Next we will describe $N$ and $g$ explicitly for the cases that $|g|>2$.

\subsubsection{$\Z_3$ orbifold construction from Niemeier lattice VOA}

First we consider the VOAs that can be obtained by a $\Z_3$ orbifold construction from Niemeier lattice VOA.   

Let $N$ be a Niemeier lattice. Then for any automorphism $g\in \Aut(V_N)$  of finite order, 
$g=\hat{\tau} \exp(2\pi i\beta(0))$ for some $\hat{\tau} \in O(\hat{N})$ and $\beta\in \mathbb{Q}\otimes_\Z N^\tau$ 
with $|g|\cdot \beta \in (N^\tau)^*$.

\medskip

\noindent \textbf{Case:} $V_1\cong A_{2,3}^6$.

In this case,  $V\cong \widetilde{V_N}(g)$, where $ N=N(A_1^{24})$, $\tau $ acts a permutation of the $24$ copies of $A_1$'s with the cycle shape $1^6 3^6$  and $\beta =\frac{1}{6}(0^{12}, \alpha^{12})$, where $\Z\alpha \cong A_1$, i.e, $\langle \alpha, \alpha\rangle =2$. In this case,  $(V_N^g)_1 \cong A_{1,3}^6U(1)^6$ and $U= \sqrt{3}L^*\cong \Span_\Z\{A_2^6, (111111)\}$. In particular, 
$O(L)= (W(A_2)\wr Sym_6). \Z_2$.  

\medskip

\noindent \textbf{Case:} $V_1\cong {A_{5,3}}{D_{4,3}}A_{1,1}^{3}$.

In this case,  $ N=N(A_5^4 D_4)$, $\tau $ acts on $A_5^4$ as a $3$-cycle and as a diagram automorphism $\varphi$ of order $3$ on $D_4$. The vector $\beta =\frac{1}3 (0,0,0, \lambda, u)$, where $\lambda= (1100-1-1) $ and $u=(1001)$. In this case,  $(V_N^g)_1 \cong A_{5,3}A_{2,3}A_{1,1}^3 U(1)^2$ and $U=\sqrt{3}L^*$ is an index $8$ overlattice of $A_5D_4(\sqrt{3}A_1)^3$.  Note $W(D_4)$ can be viewed as a subgroup of $\Stab_{\Aut(V)}(\h)$. 

\medskip

\noindent \textbf{Case:} $V_1\cong {A_{8,3}}A_{2,1}^{2}$.

In this case,  $ N=N(A_6^4)$, $\tau $ acts on $A_6^4$ as a $3$-cycle and $\beta =\frac{1}3 (0,0,0, (1^3, -1^3,0))$. In this case,  $(V_N^g)_1 \cong A_{6,3}A_{2,1}^2 U(1)^2 $ and $U=\sqrt{3}L^*$ is an index $9$ overlattice of $A_8(\sqrt{3}A_2)^2$.  

\medskip
\noindent \textbf{Case:} $V_1\cong  E_{6,3}{G_{2,1}}^3$.

In this case,  $ N=N(D_4^6)$, $\tau $ acts a diagram automorphism of order $3$ on each $D_4$ summand and $\beta =\frac{1}{3}(u, u, u, 0,0,0)$, where $u= (1,2,-1,0)\in D_4$. In this case,  $(V_N^g)_1 \cong A_{2,3}^3 G_{2,1}^3 $ and $U=\sqrt{3}L^*$ is an index $3^3$ overlattice of $E_6(\sqrt{3}A_2)^3$.  Note that $W(E_6)$ can be viewed as a subgroup of $\Stab_{\Aut(V)}(\h)$.

\medskip
\noindent \textbf{Case:} $V_1\cong {D_{7,3}}{A_{3,1}}{G_{2,1}}$.

In this case,  $N=N(D_4^6)$, $\tau $ acts as a 3-cycle on three copies of $D_4$ and as a diagram automorphism of order $3$ on 2 copies of $D_4$; $\beta =\frac{1}3(0^4,u, (1111) )$, where $u= (1,2,-1,0)\in D_4$. In this case,  $(V_N^g)_1 \cong D_{4,3} G_{2,1} A_{2,3} A_{3,1} U(1)$ and $U=\sqrt{3}L^*$ is an index $12$ overlattice of $D_7\sqrt{3}A_3\sqrt{3}A_2$. 

\medskip
\noindent \textbf{Case:} $V_1\cong E_{7,3}A_{5,1}$.

In this case,  $N=N(E_6^4)$, $\tau $ acts as a 3-cycle on three copies of $E_6$ and $\beta =\frac{1}{3}(\Lambda_1+\Lambda_2)$, where $\Lambda_1,\Lambda_2$ are fundamental weights. In this case,  $(V_N^g)_1 \cong E_{6,3}A_{5,1}U(1)$ and $U=\sqrt{3}L^*$ is an index $6$ overlattice of $E_7\sqrt{3}A_5$.

{\scriptsize
	\begin{longtable}[c]{|c|c|c|c|c|c|} 
		\caption{Orbifold construction associated with $3B$} \label{table3to8} \\
		\hline 
		No. &$\g=V_1$ & Niemeier lattice  $N$ & $\tau$ & $\beta$  & $(V_N^g)_1$ \\ \hline 
		\hline 
		$6$&$A_{2,3}^{6}$& $N(A_1^{24})$ & $1^63^6$ & $\frac{1}{6}(\alpha^{12},0^{12})$ & $A_{1,3}^6 U(1)^6$\\ 
		$17$&${A_{5,3}}{D_{4,3}}A_{1,1}^{3}$& $N(A_5^4 D_4)$ & 3-cycle$\times \varphi$ & $1/3(0,0,0, \lambda, u)$ & $A_{5,3}A_{2,3} A_{1,1}^3 U(1)^2$
		\\ 
		$27$&$A_{8,3}A_{2,1}^2$&$N(A_6^4)$ &  3-cycle   & $\frac{1}3(0,0,0,(1^3,-1^3,0) )$&$A_{6,3}A_{2,1}^2 U(1)^2$ \\ 
		$32$&$ E_{6,3}{G_{2,1}}^3$&$N(D_4^6)$ & $\varphi^{\otimes 6}$ & $\frac{1}{3}(u, u, u, 0,0,0)$   & $A_{2,3}^3 G_{2,1}^3 $ \\
		$34$&$  {D_{7,3}}{A_{3,1}}{G_{2,1}}$& $N(D_4^6)$  &  3-cycle$\times \varphi^2$ & $\frac{1}3(0^4,u, (1111) )$& $D_{4,3} G_{2,1} A_{2,3} A_{3,1} U(1)$
		\\  
		$45$&$E_{7,3}A_{5,1}$ &$N(E_6^4)$& 3-cycle & $\frac{1}{3}(\Lambda_1+\Lambda_2)$ & $E_{6,3}A_{5,1}U(1)$ \\ \hline
	\end{longtable}
}

\subsubsection{$\Z_5$ orbifold construction from Niemeier lattice VOA}
Next we consider the VOAs that can be obtained by a $\Z_5$ orbifold construction from Niemeier lattice VOA.

\medskip

\noindent \textbf{Case:} $V_1\cong A_{4,5}^2$.

In this case,  $V\cong \widetilde{V_N}(g)$, where $ N=N(A_1^{24})$, $\tau $ acts a permutation of the $24$ copies of $A_1$'s with the cycle shape $1^4 5^4$  and $\beta =\frac{1}{6}(0^{20}, \alpha^2, (2\alpha)^2)$, where $\Z\alpha \cong A_1$. In this case,  $(V_N(A_1^{24})^g)_1 \cong A_{1,5}^4 U(1)^4$ and $U= \sqrt{5}L^*\cong A_4^2$. 

\medskip

\noindent \textbf{Case:} $V_1\cong D_{6,5}A_{1,1}^2$.

In this case,  $V\cong \widetilde{V_N}(g)$, where $ N=N(A_1^{24})$, $\tau $ acts a permutation of the $24$ copies of $A_1$'s with the cycle shape $1^4 5^4$  and $\beta =\frac{1}5 (0^{22}, \alpha, 2\alpha )$.  In this case,  $(V_N(A_1^{24})^g)_1 \cong A_{1,5}^4 A_{1,1}^2U(1)^2$ and $U= \sqrt{5}L^*$ is an index $4$ overlattice of $ D_6 (\sqrt{5}A_1)^2$. 

\begin{longtable}[c]{|c|c|c|c|c|c|} 
	\caption{Orbifold construction associated with $5B$} \label{table5to4} \\
	\hline 
	No. &$\g=V_1$ & Niemeier lattice  $N$ & $\tau$ & $\beta$  & $(V_N^g)_1$ \\ \hline 
	\hline
	$9$&$ A_{4,5}^2$&$ N(A_1^{24})$ & $1^45^4$ & $\frac{1}{10}(0^{10}, \alpha^{12}, (2\alpha)^2)$ &  $A_{1,5}^4 U(1)^4$ \\
	$20$&$D_{6,5}A_{1,1}^2$&$N(A_1^{24})$ & $1^45^4$ & $\frac{1}5 (0^{22}, \alpha, 2\alpha)$& $A_{1,5}^4 A_{1,1}^2U(1)^2$ \\ \hline
\end{longtable}

\subsubsection{$\Z_7$ orbifold construction from Niemeier lattice VOA}
When $\tau$ has order $7$, there is only possible Lie algebra structure for $V_1$.

 \medskip
 
\noindent \textbf{Case:} $V_1\cong A_{6,7}$.   

\medskip 
In this case,  we choose $ N=N(A_3^8)$ and $\tau $ acts a $7$-cycle on the $8$ copies of $A_3$'s  and $\beta =\frac{1}7 (0^7, (3,-2,-1,0) )$.  In this case, $V\cong \widetilde{V_N}(g)$, $(V_N(A_1^{24})^g)_1 \cong A_{3,7}U(1)^3$ and $U= \sqrt{7}L^* \cong A_6$.  

\begin{longtable}[c]{|c|c|c|c|c|c|} 
	\caption{Orbifold construction associated with $7B$} \label{table7to5} \\
	\hline 
	No. &$\g=V_1$ & Niemeier lattice  $N$ & $\tau$ & $\beta$  & $(V_N^g)_1$ \\ \hline 
	\hline 
	$11$&$ A_{6,7}$&$N(A_3^8)$ & 7-cycle & $\frac{1}7(0^7,(3,-2,-1,0))$ & $A_{3,7} U(1)^3$  \\ \hline
\end{longtable}

\subsubsection{Remaining cases}

For the remaining case, the order of $\tau$ is not a prime.  

First we consider the cases where $\tau$ has the same frame shape as a $4C$ element in $Co_0$.

\noindent \textbf{Case:} $V_1\cong C_{7,2} A_{3,1}$.
In this case,  we choose $ N=N(A_9^2D_6)$. $\tau $ acts on $A_9^2$ as a transposition and acts as a diagram automorphism on $D_6$. The vector $\beta$ is given by  $\beta =( 0, 0, \frac{1}2(211110) )$.  Moreover,  $(V_N^g)_1 \cong A_{3,1}C_{5,2}A_{1,2} U(1)$.

\medskip
\noindent \textbf{Case:} $V_1\cong E_{6,4} A_{2,1}B_{2,1}$.

 In this case,  we choose $ N=N(A_9^2D_6)$. $\tau $ acts on $A_9^2$ as a transposition and acts as a diagram automorphism on $D_6$. The vector $\beta$ is given by  $\beta =\frac{1}8 (\,(1^5, -1^5),(1^5, -1^5), (22200) )$.  Moreover,  $(V_N^g)_1 \cong D_{5,4}A_{2,1} B_{2,1}U(1)$. 

\medskip
\noindent \textbf{Case:} $V_1\cong A_{7,4}A_{1,1}^3$.

 In this case,  we choose $ N=N(A_5^4D_4)$. $\tau $ acts as a product of two 2-cycles on $A_5^4$ times the diagram automorphism of $A_5$ on two copies of $A_5$ such that $\tau$ has order $4$. The vector $\beta$ is given by  $\beta =(\frac{1}8(1^3,-1^3)^4, \frac{1}4 (1100))$.  Moreover,  $(V_N^g)_1 \cong A_{1,1}^3A_{3,4}^2 U(1)$. 
 
 \medskip
\noindent \textbf{Case:} $V_1\cong A_{7,4}A_{1,1}^3$.

 In this case,  we choose $ N=N(A_5^4D_4)$. $\tau $ acts as a product of two 2-cycles on $A_5^4$ times the diagram automorphism of $A_5$ on two copies of $A_5$ such that $\tau$ has order $4$. The vector $\beta$ is given by  $\beta =\frac{1}8 ((1^3,-1^3)^2,0^2, (3311))$.  Moreover,  $(V_N^g)_1 \cong A_{1,1}^2C_{3,2}A_{3,4} U(1)^2$.

  \medskip
\noindent \textbf{Case:} $V_1\cong A_{3,4}^3A_{1,2}$.

 In this case,  we choose $ N=N(A_1^{24})$. $\tau $ acts as a permutation with the frame shape $1^42^24^4$. The vector $\beta$ is given by  $\beta =\frac{1}8 (1^8,2^2,0^{14})\alpha$, where $\langle \alpha| \alpha \rangle =2$.  Moreover,  $(V_N^g)_1 \cong A_{1,2}^2A_{1,4}^4 U(1)^5$.

{\scriptsize
	\begin{longtable}[c]{|c|c|c|c|c|c|c|} 
		\caption{Orbifold construction associated with $4C$} \label{4C} \\
		\hline 
		No. &$\g=V_1$ & Niemeier  $N$ & $\tau$ & $\beta$  & $(V_N^g)_1$ \\ \hline 	\hline
		$35$ &$C_{7,2}A_{3,1}$ &  $N(A_9^2 D_6)$& 2-cycle$\times (1,\delta_{A_9})\times \delta_{D_6}$ & $(0, 0, \frac{1}2(2, 1,1,1,1,0))$ & $A_{3,1}C_{5,2} A_{1,2}U(1)$\\ 
		$28$ &$E_{6,4}A_{2,1}B_{2,1}$ &  $N(A_9^2 D_6)$& 2-cycle$\times (1,\delta_{A_9})\times \delta_{D_6}$ & $\frac{1}8((1^5,-1^5),(1^5,-1^5))$ & $D_{5,4} A_{2,1}B_{2,1}U(1)$\\ 
		&  && &$+ \frac{1}4 (111000)$&\\ 
		$18$ &$A_{7,4}A_{1,1}^3$ &  $N(A_5^4 D_4)$& (2-cycle$\times (1,\delta_{A_5}))^2$ & $(\frac{1}8(1^3,-1^3)^4, \frac{1}4 (1100))$ & $A_{1,1}^3A_{3,4}^2 U(1)$\\ 
		$19$ &$D_{5,4}C_{3,2}A_{1,1}^2$ &  $N(A_5^4 D_4)$& (2-cycle$\times (1,\delta_{A_5}))^2$  & $\frac{1}8( (1^3,-1^3)^2,0^2, (3311))$ & $A_{1,1}^2 C_{3,2 }A_{3,4} U(1)^2$\\ 
		$7$ &$A_{3,4}^3A_{1,2}$&  $N(A_1^{24})$& $1^42^2 4^4$ & $\frac{1}8( 1^8, 2^2, 0^{14} ) \alpha$ & $A_{1,2} A_{1,4}^4U(1)^5$\\   		\hline
	\end{longtable}
}

\medskip

For isometries with the same frame shape as a $8E$ element, there is only possible case.

  \medskip
\noindent \textbf{Case:} $V_1\cong D_{5,8}^3A_{1,2}$.

 In this case,  we choose $ N=N(A_1^{24})$. $\tau $ acts as a permutation with the frame shape $1^22\,4 8^2$; that means $\tau$ has the same frame shape as a $8E$-element in $Co_0$. 
 The vector $\beta$ is given by  $\beta =\frac{1}{18} (3^3,1^5,0^{16})\alpha$, where $\langle \alpha| \alpha \rangle =2$.  Moreover,  $(V_N^g)_1 \cong A_{1,8}^2 U(1)^4$.  

 
 {\small
	\begin{longtable}[c]{|c|c|c|c|c|c|} 
		\caption{Orbifold construction associated with $8E$} \label{8E} \\
		\hline 
		No. &$\g=V_1$ & Niemeier  $N$ & $\tau$ & $\beta$  & $(V_N^g)_1$ \\ \hline 	\hline
		$10$ & $D_{5,8}A_{1,2}$  & $N(A_1^{24})$ &$1^2\cdot 2 \cdot 4\cdot 8^2$ &  $\frac{1}{16} (3^3,1^5, 0^{16}) \alpha $ & $A_{1,8}^2 U(1)^4$ \\\hline  
	\end{longtable} 
}

\medskip

For isometries with the same frame shape as a $6E$ element, there are two cases.

  \noindent \textbf{Case:} $V_1\cong A_{5,6}B_{2,3} A_{1,2}$.

 In this case,  we choose $ N=N(A_3^8)$ and $\tau $ acts as a product of two 3-cycles times the diagram automorphism of $A_3$ on all copies of $A_3$ on $A_3^8$. The vector $\beta$ is given by  $\beta =\frac{1}6 (\gamma_1^3, (2\gamma_1)^3,-\gamma_1,0)$, where $\gamma_1= \frac{1}4(3,-1,-1,-1)\in A_3^*$.  Moreover,  $(V_N^g)_1 \cong A_{1,2}^2A_{1,6}^2 B_{2,3} U(1)^3$.


\noindent \textbf{Case:} $V_1\cong C_{5,3}G_{2,2} A_{1,1}$.

 In this case,  we choose $ N=N(D_4^6)$. $\tau $ acts as a product of two  2-cycles times the diagram automorphism of $D_4$ on all copies of $D_4$ on $D_4^6$. The vector $\beta$ is given by  $\beta =\frac{1}6 ( (1100)^3, (1,2,-1,0),0,0) $.   Moreover,  $(V_N^g)_1 \cong A_{1,1} A_{1,3}^2 A_{1,6} G_{2,2} U(1)^2$.

{\tiny
	\begin{longtable}[c]{|c|c|c|c|c|c|} 
		\caption{Orbifold construction associated with $6E$} \label{6E} \\
		\hline 
		No. &$\g=V_1$ & Niemeier lattice  $N$ & $\tau$ & $\beta$  & $(V_N^g)_1$ \\ \hline 
		\hline
		$8$ & $A_{5,6}B_{2,3}A_{1,2}$  & $N(A_3^8)$ 
		&$\delta_{A_3}^8\times$3-cycle$^2$ 
		&$\frac{1}6 (\gamma_1^3,(2\gamma_1)^3, -\gamma_1,0) $& $A_{1,2} A_{1,6}^2 B_{2,3} U(1)^3$ \\
		$21$ & $C_{5,3}G_{2,2}A_{1,1}$ &  $N(D_4^6)$&
		$\delta_{D_4}^6\times$ 2-cycle$^2$ & $\frac{1}6((1100)^3, (12-10),0,0)$ & $A_{1,1} A_{1,3}^2 A_{1,6} G_{2,2} U(1)^2$\\ \hline
	\end{longtable}
}

  \medskip

For isometries with the same frame shape as a $6G$ element, there are also two cases.
\medskip

\noindent \textbf{Case:} $V_1\cong D_{4,12} A_{2,6}$.

 In this case,  we choose $ N=N(A_2^{12})$. $\tau $ acts as a product of a 3-cycle, a 6-cycle  times the diagram automorphism of $A_2$ on 6 copies of $A_2$ on $A_2^6$, on which the 6-cycle acts. The vector $\beta$ is given by  $\beta =\frac{1}6 (0^9, (10-1)^3)$.  Moreover,  $(V_N^g)_1 \cong A_{2,6} A_{1,12} $.  
 
\medskip
\noindent \textbf{Case:} $V_1\cong F_{4,6}A_{2,2}$.

 In this case,  we choose $ N=N(A_6^4)$. $\tau $ acts as a product of a 3-cycle times the diagram automorphism of $A_6$ on all 4 copies of $A_6$. The vector $\beta$ is given by  $\beta =\frac{1}6 (0^3,  (1^3,0, -1^3)) $.   Moreover,  $(V_N^g)_1 \cong B_{3,6} A_{2,2}U(1)$.

 {\scriptsize \begin{longtable}[c]{|c|c|c|c|c|c|} 
		\caption{Orbifold construction associated with $6G$} \label{6G} \\
		\hline 
		No. &$\g=V_1$ & Niemeier lattice  $N$ & $\tau$ & $\beta$  & $(V_N^g)_1$ \\ \hline 
		\hline
		$3$ &$D_{4,12}A_{2,6}$   &$N(A_2^{12})$& 3-cycle$\cdot$ 6-cycle$\times \delta_{A_2}^6$ & $\frac{1}{6}(0^9, (10-1)^3)$&  $A_{2,6}A_{1,12} U(1)^3$\\
		$14$ &$F_{4,6}A_{2,2}$   &$N(A_6^{4})$& 3-cycle$\times \delta_{A_6}^4$ & $\frac{1}6 (0^3,  (1^3,0, -1^3)) $ & $B_{3,6} A_{2,2}U(1)$  \\
		\hline 
	\end{longtable}
}

\medskip

For isometries with the frame shape of $10F$, there is only one possible Lie algebra.

\noindent \textbf{Case:} $V_1\cong F_{4,6}A_{2,2}$.

In this case,  we choose $ N=N(A_4^6)$ and $\tau $ acts as a product of a 5-cycle times the diagram automorphism of $A_4$ on all $6$ copies of $A_4$. The vector $\beta$ is given by  $\beta = \frac{1}{10}(0^5,  (2,1,0,-1,-2))$.   Moreover,  $(V_N^g)_1 \cong C_{2,10} U(1)^2$.  

{\small
	\begin{longtable}[c]{|c|c|c|c|c|c|} 
		\caption{Orbifold construction associated with $10F$} \label{10F} \\
		\hline 
		No. &$\g=V_1$ & Niemeier lattice  $N$ & $\tau$ & $\beta$  & $(V_N^g)_1$ \\ \hline 
		\hline
		$4$ &$C_{4,10}$  & $N(A_4^6)$ & 5-cycle$\times \delta_{A_4}^6$ & $\frac{1}{10}(0^5,  (2,1,0,-1,-2))$ &$C_{2,10} U(1)^2$\\
		\hline 
	\end{longtable}
}

\begin{rmk}\label{Invpsi}
For a root $\alpha$ of $V_1$, let $s_\alpha$ be the corresponding reflection in $W(V_1)$.  We use $\psi_\alpha$ to denote a lift of $s_\alpha$ in $\mathrm{Stab}_{\Aut(V)} (\h)$. Then $\psi_\alpha^2 = \exp(2\pi \sqrt{-1}\gamma_{(0)})$ for some $\gamma\in \h$. Up to conjugation by an element in $\{\exp(a_{(0)}\mid a\in \h\}$, we may assume $\gamma$ is fixed by $s_\alpha$ (cf. \cite[Lemma 4.5]{LSLeech}). Set $u= -\gamma/2$. 
Then $u$ is also fixed by $s_\alpha$ and  
\[
\begin{split}
(\exp(2\pi \sqrt{-1} u_{(0)}) \psi_\alpha)^2 =&   \exp(2\pi \sqrt{-1} u_{(0)})\psi_\alpha\exp(2\pi\sqrt{-1} u_{(0)})\psi_\alpha\\
 =&\exp(2\pi\sqrt{-1} u_{(0)})\psi_\alpha \exp(2\pi \sqrt{-1} u_{(0)})\psi_\alpha^{-1} 
\psi_\alpha^2\\
=& \exp(2\pi\sqrt{-1} u_{(0)}) \exp(2\pi\sqrt{-1} s_\alpha u_{(0)}) 
\exp(2\pi \sqrt{-1}\gamma_{(0)}) =1. 
\end{split}
\] 
Therefore, we may choose a lift such that $\psi_\alpha$ is an involution.  For our choices of $(N,g)$, there is always a root $\alpha$ such that $\psi_\alpha((V_N^g)_1) \neq (V_N^g)_1$. 
\end{rmk}

\section{Unitary form}\label{sec:4}

In this section, we will study the unitary form for holomorphic VOAs of central charge $24$. 
First we recall a theorem from \cite{Dlin}, which is about the unitary form for $\Z_2$ simple current extensions of unitary VOAs .  

\begin{thm}[{\cite[Theorem 3.3]{Dlin}}]\label{DLinZ2}
	Let $(V,\varphi)$ be a rational and $C_2$-cofinite unitary self-dual vertex operator algebra and $M$ a simple current irreducible $V$-module having integral weights. Assume that $M$ has an anti-linear map $\psi $ such  that $\psi(v_n w)=\varphi(v)_n \psi (w)$ and $\psi^2=id$, $(\psi(w_1),\psi(w_2))_M=(w_1,w_2)_M$ and the Hermitian form $(\ ,\ )_V$ on $V$ has the property that $(\varphi(v_1),\varphi(v_2))_V=(v_1,v_2)_V$. Then $(U,\varphi_U)$ has a unique unitary vertex operator algebra structure, where $\varphi_U:U\to U$ is the anti-linear involution defined by  $\varphi_U(v,w)=(\varphi(v),\psi(w))$, for $v\in V,w\in M$. Furthermore, $U$ is rational and $C_2$-cofinite.
\end{thm}

 By Theorem \ref{DLinZ2}, all holomorphic VOAs which can be constructed by a $\Z_2$-orbifold construction from a lattice VOA are unitary. As a consequences, we have the following result. 
 
\begin{thm}
	Let $V$ be a holomorphic VOA of central charge $24$ with the weight one Lie algebra isomorphic to one of the Lie algebras in Table \ref{T:LieZ2}. Then $V$ is unitary.   
	\end{thm} 
 
 \begin{longtable}{|c|c|c|}
 	\caption{Weight one Lie algebras of holomorphic VOAs of central charge $24$ associated with $\Z_2$ orbifolds}\label{T:LieZ2}
 	\\ \hline 
 	Class&  $\#$ of $V$& Weight one Lie algebra structures\\ \hline
 	$2A$& $17$&$A_{1,2}^{16}$, $A_{3,2}^4A_{1,1}^4$, $D_{4,2}^2B_{2,1}^4$, $A_{5,2}^2C_{2,1}A_{2,1}^2$, $D_{5,2}^2C_{2,1}A_{2,1}^2$, $A_{7,2}C_{3,1}^2A_{3,1}$,\\
 	&& $C_{4,1}^4$, $D_{6,2}C_{4,1}B_{3,1}^2$, $A_{9,2}A_{4,1}B_{3,1}$, $E_{6,2}C_{5,1}A_{5,1}$, $D_{8,2}B_{4,1}^2$, $C_{6,1}^2B_{4,1}$,\\
 	&& $D_{9,2}A_{7,1}$, $C_{8,1}F_{4,1}^2$, $E_{7,2}B_{5,1}F_{4,1}$, $C_{10,1}B_{6,1}$, $B_{8,1}E_{8,2}$\\
 	$2C$&$9$&$A_{1,4}^{12}$, $B_{2,2}^6$, $B_{3,2}^4$, $B_{4,2}^3$, $B_{6,2}^2$, $B_{12,2}$, $D_{4,4}A_{2,2}^4$, $C_{4,2}A_{4,2}^2$, $A_{8,2}F_{4,2}$\\
 	\hline 
 \end{longtable}

\subsection{Other orbifold constructions}\label{Sec:6}
Next we consider other orbifold constructions. 
The proof of the following theorem is essentially the same as \cite[Theorem 4.8]{CLS} with some necessary modifications.  For completeness, we include the proof here. 

\begin{thm}\label{Thm:uni} Let $V$ be a self-dual, simple VOA of CFT-type.
	Assume that $V$ has two commuting automorphisms $f$ and $h$ of order $p$.
	For $i,j\in\Z$, set $V^{i,j}=\{v\in V\mid f(v)=\xi^i v,\ h(v)=\xi^j v\}$, where $\xi=\exp(2\pi\sqrt{-1}/p)$.
	Set $V^i=\bigoplus_{j=0}^{p-1} V^{i,j}$.
	Assume the following:
	\begin{enumerate}[{\rm (A)}]
		\item There exists an anti-linear involution $\phi$ of $V^0$ such that $(V^0,\phi)$ is a unitary VOA;
		\item For $i\in\{1,\dots, p-1\}$, $V^i$ is a unitary $(V^0,\phi)$-module;
		\item There exists an automorphism $\psi\in \Aut(V)$ such that $\psi^{-1}f\psi=h$;
		
		\item $\psi(V^{0,0})= V^{0,0}$ and $\psi \phi\psi^{-1} =\phi $ on $V^{0,0}$;
	\end{enumerate}
	Then there exist an anti-linear involution $\Phi$ of $V$ such that $(V,\Phi)$ is a unitary VOA.
\end{thm}

\begin{rmk}\label{Rem:3^2}
	%
	Let $i,j,k,\ell\in\Z$.
	Then $\Span_\C\{ u_nv\mid u\in V^{i,j},\ v\in V^{k,\ell},\ n\in\Z\}=V^{i+k,j+\ell}$.
	Note also  that $V^{i,j}=V^{i+pk,j+p\ell}$ and $V=\bigoplus_{0\le i,j\le p-1}V^{i,j}$ is  $\Z_p^2$-graded.
	
\end{rmk}

Let $V$ be a VOA satisfying the assumptions of Theorem \ref{Thm:uni}. 
Let $(\cdot,\cdot)_{V^0}$ be the positive-definite invariant Hermitian form on $V^0$ normalized so that $(\1,\1)_{V^0}=1$.
Let $\langle\cdot,\cdot\rangle$ be the normalized symmetric invariant bilinear form on $V$ such that $\langle\1,\1\rangle=1$.
Note that $(u,v)_{V^0}=\langle u,\phi(v)\rangle$ for $u,v\in V^{0}$ (cf. Remark \ref{Rem:inv}).
By the assumption (C), $\psi(V^0)=V^{0,0}\oplus V^{1,0}\oplus \cdots \oplus V^{p-1,0}$ is also a unitary VOA with the anti-linear automorphism $\psi\phi \psi^{-1}$ and a positive-definite invariant Hermitian form defined by 
\begin{equation}
	(a,b)_{\psi(V^0)}=(\psi^{-1}(a),\psi^{-1}(b))_{V^0}\quad \text{for}\quad a,b\in \psi(V^0).\label{Eq:p()}
\end{equation}
Note that $\psi\phi\psi^{-1}=\phi$ on $V^{0,0}$ by Assumption (D).

By Lemma \ref{uniqueF}, a positive-definite invariant Hermitian form on the unitary $(V^{0,0},\phi)$-module $V^{i,0}$ is unique up to scalar for each $i=1, \dots, p-1$.
We may choose a positive-definite invariant Hermitian form $(\cdot,\cdot)_{V^i}$ on $V^i$ so that 
\begin{equation}
	(u,v)_{V^i}=(u,v)_{\psi(V^0)}\quad \text{for}\quad u,v\in V^{i,0}.\label{Eq:choice}
\end{equation}
By Lemma \ref{Rem:specify}, there exists an anti-linear bijective map $\Phi^i:V^i\to V^{p-i}$ such that $\Phi^i(a_nv)=\phi(a)_n\Phi^i(v)$ for $a\in V^0$ and $v\in V^i$ and 
\begin{equation}
	(u,v)_{V^i}=\langle u,\Phi^i(v)\rangle\quad \text{for}\quad u,v\in V^i.\label{Eq:()<>}
\end{equation}
By \eqref{Eq:p()}, \eqref{Eq:choice} and \eqref{Eq:()<>}, for  any $u,v\in V^{i,0}$, we have 
$$\langle u,\Phi^i(v)\rangle=(u,v)_{V^i}=(\psi^{-1}(u),\psi^{-1}(v))_{V^0}=\langle \psi^{-1}(u),\phi\psi^{-1}(v)\rangle=\langle u,\psi\phi\psi^{-1}(v)\rangle.$$
Hence
\begin{equation}
	\psi\phi\psi^{-1}=\Phi^i\quad \text{on}\quad V^{i,0}.\label{Eq:comm}
\end{equation}

Since the order of $\phi$ is $2$, both the composition maps $\Phi^{p-i}\circ\Phi^i$ and  $\Phi^{i}\circ\Phi^{p-i}$ are the identity map on $V^{i,0}$ by \eqref{Eq:comm}.
Viewing $V^{p-i}$ as an irreducible unitary $(V^0,\phi)$-module, we have $\Phi^{p-i}=(\Phi^i)^{-1}$ on $V^{p-i}$ by the same argument as in the proof of Lemma \ref{uniqueF}.

Now, we define the anti-linear map $\Phi:V\to V$ so that 
\begin{align*}
	\Phi(u) =\begin{cases} \phi(u)& \text{for}\ u\in V^0,\\
		\Phi^i(u)& \text{for}\ u\in V^i,\ i=1,\dots, p-1, 
	\end{cases}
\end{align*}
and the positive-definite Hermitian form $(\cdot,\cdot)$ on $V$ by
\begin{align*}
	(u,v)=\begin{cases}
		(u,v)_{V^i} & \text{ if } u,v\in V^i,\ i=0,1,\dots, p-1,\\
		0 & \text{ if } u\in V^i,\ v\in V^j,\ i\neq j.
	\end{cases}
\end{align*}
Clearly, $\Phi$ is bijective and $\Phi\circ \Phi$ is the identity of $V$.
We will show that $(V,\Phi)$ is unitary. 

\begin{lem}\label{Lem:<>()}
	\begin{enumerate}[{\rm (1)}]
		\item For $i,j\in\{0,1,\dots, p-1\}$, $\Phi(V^{i,j})=V^{p-i,p-j}$.
		\item For $u,v\in V$, $(u,v)=\langle u,\Phi(v)\rangle$.
	\end{enumerate}
\end{lem}
\begin{proof} Clearly, $(V^{0,0},\phi)$ is a simple unitary VOA and $V^{i,j}$ is an irreducible unitary $(V^{0,0},\phi)$-module.
	Hence by Lemma \ref{Rem:specify} (2) and Lemma \ref{uniqueF}, the map $\Phi$ sends $V^{i,j}$ to the submodule of $V^{p-i,p-j}$ isomorphic to the contragredient module of $V^{i,j}$, which is $V^{p-i,p-j}$ by Remark \ref{Rem:3^2} (4).
	Hence we obtain (1).
	
	Let $u\in V^i$, $v\in V^j$ with $i\neq j$.
	Clearly $i- j\neq 0 \mod p$.
	By (1), we have $\Phi(v)\in V^{p-j}$.
	Since the contragredient module of $V^i$ is not isomorphic to $V^{p-j}$, we have $\langle u,\Phi(v)\rangle=0$.
	By \eqref{Eq:()<>} and the definition of the form $(\cdot,\cdot)$, we obtain (2).
\end{proof}

\begin{prop}\label{Prop:anti} The anti-linear map $\Phi$ is an anti-linear involution of $V$.
\end{prop}
\begin{proof} Since $\phi$ is an anti-linear automorphism of $V^0$, $\Phi$ fixes the vacuum vector and the conformal vector of $V$.
	Since $(V^0,\phi)$ is unitary, the equation 
	\begin{equation}
		\Phi(u_nv)=\Phi(u)_n\Phi(v)\label{Eq:anti}
	\end{equation} holds for $u,v\in V^0$ and $n\in\Z$.
	By the definition of $\Phi^i$ for $i=1,2$, \eqref{Eq:anti} holds for $u\in V^0$ and $v\in V^i$.
	By the skew symmetry, we have 
	\[
	u_{n}v = (-1)^{n+1} v_{n}u +\sum_{i\geq 1} \frac{(-1)^{n+i+1}}{i!} L(-1)^i (v_{n+1} u)
	\]
	for $u,v\in V$ and $n\in\Z$.
	Hence the equation \eqref{Eq:anti} also holds for $u\in V^i$ and $v\in V^0$.
	
	Let $x\in V^{0,j}$, $y\in V^{i,0}$ and $u\in V^{k,\ell}$.
	By Borcherds' identity, for $r,q\in\Z$,
	$$(x_{r}y)_{q}u=\sum_{i=0}^\infty(-1)^i\binom{r}{i}\left(x_{r-i}(y_{q+i}u)-(-1)^ry_{q+r-i}(x_{i}u)\right).$$
	By the assumptions on $x$ and $y$ and the identity above, 
	we have $$\Phi((x_{r}y)_{q}u)=(\Phi(x)_{r}\Phi(y))_{q}\Phi(u)=\Phi(x_{r}y)_{q}\Phi(u).$$
	By Remark \ref{Rem:3^2} (2), we obtain $\Phi(u_{n}v)=\Phi(u)_{n}\Phi(v)$ for all $x,y\in V$ and $n\in\Z$.
\end{proof}

By Lemma \ref{Lem:<>()} (2), the invariant property of $\langle\cdot,\cdot\rangle$ and Proposition \ref{Prop:anti}, we obtain the following proposition:

\begin{prop}\label{Prop:Unitary}
	The positive-definite Hermitian form $(\ , \ )$ on $V$ satisfies the invariant property for $(V,\Phi)$.
\end{prop} 



Combining Propositions \ref{Prop:anti} and \ref{Prop:Unitary}, we have proved Theorem \ref{Thm:uni}.

\subsection{Unitary forms for holomorphic VOAs of central charge $24$}
As we discussed in Section \ref{NieOrb}, every holomorphic VOA of central charge $24$ with $V_1\neq 0$ can be constructed by a single orbifold construction from a Niemeier lattice VOA.  

Let $(N,g)$ be pair of a Niemeier lattice and an automorphism of $V_N$ as described in Section \ref{NieOrb} such that 
$V\cong \widetilde{V_N}(g)$. Then   
\[
V= V_N^g \oplus V_N[g]_0 \oplus \cdots \oplus  V_N[g^{p-1}]_0, 
\]
where $V_N[g^i]$ denotes the irreducible $g^i$-twisted module of $V_N$. 

Let $L$ be the even lattice such that $V_L\cong \mathrm{Com}_V ( \mathrm{Com}_V (M(\h)))$, where $\h$ is a Cartan subalgebra of $V_1$ and suppose $g=\hat{\tau}\exp(2\pi i\beta(0))\in \Aut(V_N)$.  
Then $L\cong N_\beta^\tau$ and $V_N^g> V_L\otimes V_{\Lambda_\tau}^{\hat{\tau}}$. 

Set  
\[
V_N=\bigoplus_{\lambda+N^\tau \in (N^\tau)^*/N^\tau}
 V_{\lambda+N^\tau}\otimes V_{\lambda'+N_\tau}. 
 \] 
Then 
 \[
 V_N^g =  \bigoplus_{\lambda+N^\tau \in (N^\tau)^*/N^\tau}
\left( V_{\lambda+N^\tau}\otimes V_{\lambda'+N_\tau} \right)^g= \bigoplus_{\lambda+L \in (N^\tau)^*/L}
 V_{\lambda+L}\otimes W_{\lambda} < V,  
 \] 
where $W_{\lambda}$, $\lambda+L \in (N^\tau)^*/L$, are irreducible $V_{N_\tau}^{\hat{\tau}}$-modules and eigenspaces of $\hat{\tau}$ on $V_{\lambda'+N_\tau} $.  
 
Define $f\in \Aut(V)$ so that  $f$ acts on $V_N[g^i]_0$ as a multiplication of the scalar $\xi^i$. 
Then $V^f =V_N^g$ and there is a $\gamma\in \mathbb{Q}\otimes_\Z N^\tau$ such that $\langle \gamma|\beta \rangle \notin \Z$ and $f=\exp(2\pi i \gamma(0))$.  As we discussed in Section \ref{NieOrb}, the Lie subalgebra $(V_N^g)_1$ is proper subalgebra of $V_1$ and $(V_N^g)_1$  is non-abelian. By Lemma  \ref{Lem:4.1}, Proposition \ref{prop4.4} and Remark \ref{Invpsi}, there is a root of $V_1$ and a lift $\psi_\alpha\in \Stab_{\Aut(V)}(V_{L}\otimes W)$ of a reflection $s_\alpha\in W(V_1)$ such that $\psi_\alpha((V_N^g)_1) \neq (V_N^g)_1$ and $\psi_\alpha^2=1$. 
For simplicity, we use $w$ and $\psi$ to denote $s_\alpha$ and $\psi_\alpha$, respectively.

Define $h=\psi f\psi^{-1}$.  Then $h=\exp(2\pi i w(\gamma)(0))$ and it is clear that both $f$ and $h$ fix  $V_L\otimes V_{\Lambda_\tau}^{\hat{\tau}} $ point-wisely. Since all irreducible modules for $V_L\otimes V_{\Lambda_\tau}^{\hat{\tau}} $ are simple current modules, the subgroup of $\Aut(V)$ that fixes $V_L\otimes V_{\Lambda_\tau}^{\hat{\tau}} $ point-wisely is a finite abelian group. In particular,  $[f,h]=1$.  Moreover, we have 
$$V^{0,0} =  V^{<f,h>}= \bigoplus _{\lambda+L \in J/L} V_{\lambda+L} \otimes W_{\lambda},$$ 
where $J= \{ \lambda\in L^*\mid \langle \lambda, \gamma \rangle \in \Z, \langle \lambda, w(\gamma) \rangle \in \Z\}$. 

\begin{lem}
We have $w(J)=J$ and $\psi(V^{0,0})=V^{0,0}$. 
\end{lem}

\begin{proof}
	Let $\lambda\in J$. Then $\langle w(\lambda)|w(\gamma)\rangle =\langle \lambda| \gamma\rangle\in \Z$ and $\langle w(\lambda)|\gamma \rangle =\langle w^2(\lambda)|w(\gamma)\rangle= \langle \lambda|w(\gamma)\rangle\in \Z$. Thus $w(\lambda)\in J$ for any $\lambda\in J$ and we have the desired result. 
\end{proof}

\begin{lem}
Let $X$ be a sublattice of $N$ such that $P_0(X)=J$.  Then  $V^{0,0} < V_X$ and $\psi$ can be considered as a lift of an isometry of $X$ in $\Aut(V_X)$. In particular, we have $\psi\phi \psi^{-1} =\phi $ on $V^{0,0} < V_X$.  
\end{lem}

\begin{proof}
We first note that $J= (N^\tau)^* \cap w((N^\tau)^*) > L$. Then 
\[
V_X = \bigoplus_{\lambda+L\in J/L} V_{\lambda+L}\otimes V_{\lambda'+N_\tau}.   
\]
Since $W_{\lambda} < V_{\lambda'+N_\tau}$, we have $V^{0,0} < V_X$. Since $w(X) <X$, $w$ defines an isometry of $X$ and thus $\psi \phi \psi^{-1} = \phi$ on $V_X$.  We have $\psi \phi \psi^{-1} = \phi$ on $V^{0,0}$ as desired.
\end{proof}

Therefore, $V$,$f$ and $h$ satisfy the conditions in Theorem \ref{Thm:uni} and the main theorem follows. 
\begin{thm}
Any (strongly regular) holomorphic vertex operator algebra of central charge $24$ and with non-trivial weight one subspace is unitary. 
\end{thm}

\newcommand{\etalchar}[1]{$^{#1}$}
\providecommand{\bysame}{\leavevmode\hbox to3em{\hrulefill}\thinspace}
\providecommand{\MR}{\relax\ifhmode\unskip\space\fi MR }
\providecommand{\MRhref}[2]{%
  \href{http://www.ams.org/mathscinet-getitem?mr=#1}{#2}
}
\providecommand{\href}[2]{#2}

\end{document}